\documentclass[12pt]{article}
\usepackage{graphicx}
\usepackage{amsmath,amsthm,amssymb}
\usepackage{authblk,epstopdf}

\def\a{\alpha}
\def\ba{\bar\alpha}

\DeclareMathOperator\re{{Re}}

\numberwithin{equation}{section}
\newtheorem{thm}{Theorem}[section]
\newtheorem{lem}[thm]{Lemma}
\newtheorem{prop}[thm]{Proposition}

\begin{document}
\title{Supercritical Hopf Bifurcation of Cooperative Predation}

\author[1]{Srijana Ghimire}
\author[1]{Xiang-Sheng Wang\thanks{Corresponding author. Email: xswang@louisiana.edu}}

\affil[1]{Department of Mathematics, University of Louisiana at Lafayette, Lafayette, LA 70503, USA}

\date{}

\maketitle

\begin{abstract}
In this work, we conduct a rigorous analysis on the dynamics of a predator-prey model with cooperative predation.
From the root classification of an algebraic equation, we derive existence criteria of the positive equilibria.
By Jacobian matrix and central manifold theory, we find critical conditions under which the positive equilibria are locally asymptotically stable or unstable.
We also use a careful computation to obtain a concise and explicit formula for the first Lyapunov coefficient. Especially, we prove that the Hopf bifurcation induced by the cooperative predation is always supercritical, which means that the sustained oscillations near the Hopf bifurcation points are locally asymptotically stable.
\end{abstract}


\bigskip


\section{Introduction}

Cooperation plays a significant role in many biological species \cite{Dugatkin97}.
Cooperative hunting is an important subject in Phylogenetics \cite{Beauchamp14}, and it is essential for carnivores \cite{Macdonald83} and other predators \cite{PR88}.
For example, a larger group of Yellowstone wolves is more likely to capture their prey, bison \cite{MTSS14}. Harris' hawks in New Mexico will hunt cooperatively during the nonbreeding season so as to improve their capture success \cite{Bednarz88}. {\it D. discoideum}, a soil amoeba that lives mostly as single cells, will develop social cooperation if it is starving \cite{CF04,LP11}.
Cooperative hunting is also observed in ants \cite{Moffett88}, African wild dogs \cite{Creel95}, avian predators \cite{Hector86},
lions \cite{Packer90,Scheel91}, spiders \cite{Uetz91}, wild chimpanzees \cite{Boesch94} and wolves \cite{Schmidt97}.

Lotka-Volterra system is a standard mathematical model for predation and it has been extensively used in predator-prey model ever since Lotka and Volterra conducted two independent studies of parasite invasion \cite[p. 88]{Lo25} and fishery data \cite{Vo26}, respectively.
This system was extended by \cite{AH17} to investigate cooperative hunting and further generalized in \cite{JZL18} to include Allee effects in the prey.
In \cite{Freedman01}, a general population model of cooperative predation was proposed and some conditions for existence of positive equilibrium were obtained. However, since the predation function in \cite{Freedman01} was too general, it seems impossible to conduct stability and bifurcation analysis for the model system. In this work, we will provide a detailed analysis on the following model proposed in \cite{AH17}.
\begin{align}
  U'(T)&=BU(T)[1-U(T)/K]-PU(T)V(T)-QU(T)V(T)^2,\label{U}\\
  V'(T)&=C[PU(T)V(T)+QU(T)V(T)^2]-DV(T),\label{V}
\end{align}
where $U(T)$ and $V(T)$ denote the densities of prey and predator at time $T$. The prey has a logistic growth rate with carrying capacity $K$.
The predation includes a bilinear function $PUV$ which accounts for the mass action and a cooperative hunting function $QUV^2$. The constant $C\in(0,1]$ is the rate of energy conversion and the constant $D>0$ is the per capita death rate of the predator.
As mentioned in \cite{AH17}, the above model differs from the one in \cite{Berec10} where a Holling type II functional response was chosen to generate sustained oscillation even without cooperative predation. For a unified mechanistic study of predation rates, we refer to \cite{Cosner99} and references therein.
Numerical simulations \cite{AH17} indicate that large cooperative predation rate may enhance the survival of the predator and induce bistability and oscillations. However, the numerical results are only valid for a specific set of parameter values. It is still unclear when a positive equilibrium exists and whether it is stable for a general set of parameter values. This motivates us to find existence criteria and stability conditions for the positive equilibria. We will provide a rigorous proof of the numerical observations such as bistability phenomenon in \cite{AH17}. Moreover, we will conduct local Hopf bifurcation analysis and calculate the first Lyapunov coefficient. From our analysis, we will demonstrate the Hopf bifurcation (if exists) is always supercritical; namely, the periodic solutions bifurcated from the Hopf points are always locally asymptotically stable.

We organize the rest of this work as follows.
In Section 2, we introduce dimensionless variables and parameters, and also present some preliminary results.
In Section 3, we develop existence criteria of the positive equilibria.
In Section 4, we investigate the stability of positive equilibria.
In Section 5, we conduct local Hopf bifurcation analysis and calculate the first Lyapunov coefficient.
In Section 6, we state our main theorem for the original system in terms of non-scaled parameters.
In Section 7, { we conduct numerical simulations to illustrate and verify our theoretical results.
In Section 8,} we give a brief conclusion of our work and propose an open problem.

\section{Nondimensionalization and preliminaries}
To simplify our analysis, we introduce the following dimensionless variables
\begin{equation}
  u={U\over K},~~v={v\over CK},~~t=DT.
\end{equation}
We then rewrite \eqref{U}-\eqref{V} as an equivalent system
\begin{align}
  u'(t)&=bu(t)[1-u(t)]-pu(t)v(t)-qu(t)v(t)^2,\label{u}\\
  v'(t)&=pu(t)v(t)+qu(t)v(t)^2-v(t),\label{v}
\end{align}
where
\begin{align}\label{bpq}
  b={B\over D},~~p={CPK\over D},~~q={C^2QK^2\over D}.
\end{align}
It is noted that the { dimensionless} parameter $p$ is the same as the basic reproduction number of the predator.
The Jacobian matrix for the system \eqref{u}-\eqref{v} linearized about an equilibrium is calculated as
\begin{align}\label{J}
  J(u,v)=\begin{pmatrix}
    b(1-{ 2u})-pv-qv^2&-pu-2quv\\
    pv+qv^2&pu+2quv-1
  \end{pmatrix}.
\end{align}
We have the following preliminary results.
\begin{prop}\label{prop-E0}
  The system \eqref{u}-\eqref{v} always possesses a trivial equilibrium $E_0=(0,0)$ and a predator-free equilibrium $E_1=(1,0)$.
  The trivial equilibrium $E_0$ is always unstable. The predator-free equilibrium $E_1$ is locally asymptotically stable if $p<1$ and unstable if $p>1$.
  For the critical case $p=1$, the predator-free equilibrium $E_1$ is locally asymptotically stable if $q\le1/b$ and unstable if $q>1/b$.
\end{prop}
\begin{proof}
  The Jacobian matrix corresponding to $E_0$ is
  $$J_0=\begin{pmatrix}
    b&0\\0&-1
  \end{pmatrix}.$$
  Since $J_0$ has a positive eigenvalue $b$, the trivial equilibrium $E_0$ is always unstable.
  The Jacobian matrix corresponding to $E_1$ is
  $$J_1=\begin{pmatrix}
    -b&-p\\0&p-1
  \end{pmatrix}.$$
  If $p<1$, then $J_1$ has two negative eigenvalue $-b$ and $p-1$, and hence $E_1$ is locally asymptotically stable.
  If $p>1$, then $J_1$ has a positive eigenvalue $p-1$, which implies that $E_1$ is unstable.
  For the critical case $p=1$, we factorize the Jacobian matrix as
  $$J_1=\begin{pmatrix}
    -b&-1\\0&0
  \end{pmatrix}=\begin{pmatrix}
    1&-1/b\\0&1
  \end{pmatrix}\begin{pmatrix}
    -b&0\\0&0
  \end{pmatrix}\begin{pmatrix}
    1&1/b\\0&1
  \end{pmatrix}.$$
  Introduce new state variables
  \begin{align*}
    \begin{pmatrix}
      x\\y
    \end{pmatrix}=\begin{pmatrix}
    1&1/b\\0&1
  \end{pmatrix}
    \begin{pmatrix}
      u-1\\v
    \end{pmatrix}=
    \begin{pmatrix}
      u-1+v/b\\v
    \end{pmatrix}.
  \end{align*}
  It then follows that $u=1+x-y/b$ and $v=y$. Moreover, \eqref{v} becomes
  $$y'=y[x-y/b+q(1+x-y/b)y].$$
  Following \cite[Section 2.1]{Wi90}, we restrict the above equation on the center manifold $x=O(y^2)$ and find
  $$y'=(q-1/b)y^2+O(y^3).$$
  If $q<1/b$, then $y=0$ is locally asymptotically stable for the above equation, and consequently, $E_1$ is locally asymptotically stable for the original system \eqref{u}-\eqref{v}.
  On the other hand, if $q>1/b$, then $y=0$ is unstable for the above equation, and hence, $E_1$ is unstable for the original system \eqref{u}-\eqref{v}.
  Finally, we consider the case $p=1$ and $q=1/b$. Again, we restrict the system \eqref{u}-\eqref{v} on the center manifold $x=c_2y^2+O(y^3)$.
  A simple calculation gives
  \begin{align*}
    x'&=b(x-qy+1)(qy-x)-(1+x-qy)y-q(1+x-qy)y^2
    \\&=-(bc_2+q)y^2+O(y^3),\\
    y'&=(1+x-qy)y+q(1+x-qy)y^2-y
    \\&=(c_2-q^2)y^3+O(y^4).
  \end{align*}
  Coupling the second equation with $x=c_2y^2+O(y^3)$ implies that $x'=O(y^4)$. Hence, the coefficient of $y^2$ on the right-hand side of the first equation vanishes; namely, $c_2=-q/b=-q^2$. Substituting this back into the second equation yields
  $$y'=-2q^2y^3+O(y^4).$$
  Thus, $y=0$ is locally asymptotic stable for the above equation and $E_1$ is locally asymptotically stable for the original system \eqref{u}-\eqref{v}.
  This completes the proof.
\end{proof}

\section{Existence of positive equilibria}
An equilibrium of \eqref{u}-\eqref{v} is a solution to the algebraic system
\begin{equation}
  bu(1-u)=puv+quv^2=v.
\end{equation}
Assuming $v>0$ and eliminating $v$ from the system gives
\begin{equation}
  pu+qbu^2(1-u)=1.
\end{equation}
For convenience, we set $s=1/u$ and rewrite the above equation as
\begin{equation}\label{f}
  f(s):=s^3-ps^2-qb(s-1)=0.
\end{equation}
There is a one-to-one corresponding of the positive equilibria $(u,v)$ to the roots of $f(s)$ greater than $1$:
\begin{equation}
  u=1/s,~~v=b(s-1)/s^2.
\end{equation}
We then have the following lemma.
\begin{lem}\label{lem-f}
  Let $f(s)$ be given in \eqref{f}{.}
  If $p>1$, then $f(s)$ has a unique root $s_+>1$.
  If $p<1$, then $f(s)$ has exactly two roots $s_\pm>1$ when $q>q_0$ and no root greater than $1$ when $q<q_0$,
  where
  \begin{equation}\label{q0}
    q_0:={(9-p)\sqrt{(9-p)(1-p)}+27-18p-p^2\over8b}.
  \end{equation}
  The two roots $s_\pm>1$ coincide when $q=q_0$.
  For the critical case $p=1$, we have $q_0=1/b$ and $f(s)$ has exactly two roots $s_\pm>1$ when $q>q_0$ and no root greater than $1$ when $q\le q_0$.
\end{lem}
\begin{proof}
  Note that $f'(s)=3s^2-2ps-qb$ has two real roots
  $$s_1={p-\sqrt{p^2+3qb}\over3}<0,~~s_2={p+\sqrt{p^2+3qb}\over3}>0.$$
  Since $f(0)=qb>0$, $f(s)$ has exactly one negative root.
  If $p>1$, then $f(1)=1-p<0$, which implies that $f(s)$ has a second root $s_-\in(0,1)$ and the third root $s_+>1$.
  If $p<1$, then $f(1)=1-p>0$. $f(s)$ has two roots $s_\pm>1$ if and only if $f'(1)<0$ and $f(s_2)<0$.
  From $f'(1)=3-2p-qb<0$, we have $qb>3-2p$. To solve $f(s_2)<0$, we first note that $f'(s_2)=0$ and hence,
  \begin{align*}
    f(s_2)={s_2(2ps_2+qb)\over3}-ps_2^2-qb(s_2-1)
    =-{p(2ps_2+qb)\over9}-{2qbs_2\over3}+qb
  \end{align*}
  The inequality $f(s_2)<0$ is the same as
  \begin{align*}
    {qb(9-p)\over2p^2+6qb}<s_2={p+\sqrt{p^2+3qb}\over3};
  \end{align*}
  namely,
  \begin{align*}
    {27qb-9pqb-2p^3\over2p^2+6qb}<\sqrt{p^2+3qb}.
  \end{align*}
  Since $qb>3-2p$, the left-hand side of the above inequality is positive.
  By squaring both sides of the above inequality, we obtain from a simple calculation that $q>q_0$.
  It is easy to verify that $q_0b>3-2p$. Therefore, $f(s)$ has two roots $s_\pm>1$ if and only if $q>q_0$.
  On the other hand, if $q<q_0$, then either $qb<3-2p$ (in this case $f'(1)<0$) or $f(s_+)>0$; in either case, $f(s)$ does not have any root greater than $1$.
  Moreover, if $q=q_0$, then $f(s_2)=f'(s_2)=0$ and $f(s)$ has a double root $s_\pm=s_2>1$.
  Finally, we consider the critical case $p=1$. Since $f(1)=0$, $f(s)$ has a root $s_+>1$ if and only if $f'(1)<0$; namely, $q>q_0=1$.
  This completes the proof.
\end{proof}
A direct application of Lemma \ref{lem-f} is the following existence conditions of positive equilibria.
\begin{prop}\label{prop-E}
  Let $q_0$ be given as in \eqref{q0}.
  The system \eqref{u}-\eqref{v} possesses a unique positive equilibrium $E_+=(u_+,v_+)$ if $p>1$,
  and exactly two positive equilibria $E_\pm=(u_\pm,v_\pm)$ if $p\le 1$ and $q>q_0$, and no positive equilibrium if $p\le 1$ and $q<q_0$, and one positive equilibrium $E_+=E_-$ if $p<1$ and $q=q_0$, and no positive equilibrium if $p=1$ and $q=q_0$.
  Here, $u_\pm=1/s_\pm$ and $v_\pm=b(s_\pm-1)/s_\pm^2$ with $s_+\ge s_-$ being the positive roots (if exist) of $f(s)$ defined in \eqref{f}.
\end{prop}

\section{Stability of positive equilibria}
To investigate the stability of a positive equilibrium $E=(u,v)$ (if exists),
we calculate the Jacobian matrix in \eqref{J}.
On account of $bu(1-u)=pu v+qu v^2=v$, we obtain
\begin{align}
  J:=J(u,v)=\begin{pmatrix}
    -bu&-1-qu v\\
    b-bu&qu v
  \end{pmatrix}.
\end{align}
Furthermore, the trace $tr(J)=u(qv-b)$ and the determinant
\begin{align}\label{J-pm}
  det(J)=b[1-u+qu v(1-2u)]=u^2v[s^3+bq(s-2)]=uvf'(s),
\end{align}
where $s=1/u$ and $f(s)=s^3-ps^2-bq(s-1)$; see \eqref{f}.
The equilibrium $E$ is locally asymptotically stable if $tr(J)<0$ and $det(J)>0$, and unstable if $tr(J)>0$ or $det(J)<0$.
Recall from Proposition \ref{prop-E} that if $p\le1$ and $q>q_0$, then there exists two positive equilibria $E_\pm=(u_\pm,v_\pm)$ with $f'(s_+)>0>f'(s_-)$ where $s_\pm=1/u_\pm$ and $1<s_-<s_+$.
Hence, the Jacobian matrix $J_-:=J(u_-,v_-)$ corresponding to $E_-$ has a negative determinant, which implies that $E_-$ is unstable.
Similarly, the Jacobian matrix $J_+:=J(u_+,v_+)$ corresponding to $E_+$ has a positive determinant.
The stability of $E_+$ is then determined by the sign of $tr(J_+)=u_+(qv_+-b)$.
To find the critical bifurcation value, we set $tr(J_+)=0$; namely $qv_+=b$. This together with $bu_+(1-u_+)=pu_+v_++qu_+v_+^2=v_+$ implies that
$u_+=1/(p+b)$ and
$$q={b\over v_+}={1\over u_+(1-u_+)}={(p+b)^2\over p+b-1}.$$
It is thus reasonable to define
\begin{equation}\label{qh}
  q_h:={(p+b)^2\over p+b-1}
\end{equation}
when $p+b>1$.
We will prove that the positive equilibrium $E_+$ switches its stability when $q$ crosses the bifurcation point $q_h$.
\begin{prop}\label{prop-s+}
  If $p>1$, then the unique positive equilibrium $E_+=(u_+,v_+)$ is locally asymptotically stable when $q<q_h$ and unstable when $q>q_h$, where $q_h$ is defined in \eqref{qh}.
\end{prop}
\begin{proof}
  We only need to show that $qv_+-b$ has the same sign as $q-q_h$.
  It is obvious that $qv_+=b$ if and only if $q=q_h$.
  Now, we assume that { $q<q_h$.
  It follows that
  $$f(p+b)=b(p+b-1)(q_h-q)>0.$$
  Since $s_+=1/u_+$ is the unique root of $f(s)$ in $(1,\infty)$, we have
  $s_+<p+b$ and
  $$qv_+=s_+-p<b.$$
  }In a similar manner, one can obtain from
  { $q>q_h$ that $f(p+b)<0$, which implies that $s_+>p+b$ and $qv_+=s_+-p>b$.}
  This completes the proof.
\end{proof}
For the case ${ p}\le1$, a pair of positive equilibria $E_\pm$ exist if and only if $q>q_0$.
According to the argument at the beginning of this section, $E_-$ is always unstable if $q>q_0$.
Similar as in the proof of Proposition \ref{prop-s+}, the stability of $E_+$ switches as $q$ cross the bifurcation value $q_h$, provided that
$q_h>q_0$. Recall the definitions of $q_0$ and $q_h$ in \eqref{q0} and \eqref{qh}. The condition $q_h>q_0$ is equivalent with
\begin{equation}
  b>{3(1-p)+\sqrt{(1-p)(9-p)}\over4}.
\end{equation}
We then have the following proposition.
\begin{prop}\label{prop-s-}
  Assume $p\le1$ and $q>q_0$. $E_-$ is always unstable. If $b>[3(1-p)+\sqrt{(1-p)(9-p)}]/4$ and $q<q_h$, then $E_+$ is locally asymptotically stable.
  If $b<[3(1-p)+\sqrt{(1-p)(9-p)}]/4$ or $q>q_h$, then $E_+$ is unstable.
\end{prop}
\begin{proof}
  Note that $1<s_-<s_+$ and $f'(s_-)<0<f'(s_+)$. The Jacobian matrix associated with $E_-$ has a negative determinant, which implies that $E_-$ is always unstable. On the other hand, the determinant of $J_+$ is positive, where $J_+=J(u_+,v_+)$ is the Jacobian matrix associated with $E_+$.
  If $b>[3(1-p)+\sqrt{(1-p)(9-p)}]/4$, then $b>1-p$ and $q_h$ is well defined. Using a similar argument as in the proof of Proposition \ref{prop-s+}, one can show that the trace of $J_+$ has the same sign as $q-q_h$. Hence, $E_+$ is locally asymptotically stable when $q<q_h$ and unstable when $q>q_h$.

  Finally, if $b<[3(1-p)+\sqrt{(1-p)(9-p)}]/4$, it follows from $f'(s_+)>0$ and $f(s_+)=0$ that $s_+>[p+3+\sqrt{(1-p)(9-p)}]/4>p+b$. Consequently, $qv_+=s_+-p>b$; namely, $tr(J_+)>0$. This implies that $E_+$ is unstable. The proof is complete.
\end{proof}

\section{First Lyapunov coefficient}

The argument in the previous section implies that $q_h$ is a bifurcation value for the stability of $E_+$ and a pair of purely imaginary eigenvalues of the Jacobian matrix $J_+$ appears as $q$ crosses $q_h$ from left to right; i.e., Hopf bifurcation occurs at $q=q_h$. To obtain the direction of Hopf bifurcation and investigate the stability of the periodic solutions bifurcated from the Hopf bifurcation point $q_h$, we shall calculate the first Lyapunov coefficient.

Throughout this section, we assume that either (i) $p>1$ or (ii) $p\le1$ and $b>[3(1-p)+\sqrt{(1-p)(9-p)}]/4$. In each case $q_h$ is well defined and $E_+$ switches stability as $q$ crosses the bifurcation value $q_h$. Now, we fix $q=q_h$. The Jacobian matrix $J_+$ has a pair of purely imaginary roots $\pm iw$, where $w^2=det(J_+)$. For convenience, we denote $s=s_+=1/u_+$ and $z:=bu_+$. It then follows that
\begin{equation}\label{sz}
u_+={1\over s},v_+={z(s-1)\over s},b=zs,p=s(1-z),q={s^2\over s-1}={s^2z(z+1)\over w^2+z^2}.
\end{equation}
Moreover, we can factorize the Jacobian matrix as
\begin{align*}
  J_+&=\begin{pmatrix}
    -z&-1-z\\
    z(s-1)&z
  \end{pmatrix}
  \\&=\begin{pmatrix}
    1+z&1+z\\
    -z-iw&-z+iw
  \end{pmatrix}\begin{pmatrix}
    iw&0\\
    0&-iw
  \end{pmatrix}\begin{pmatrix}
    1+z&1+z\\
    -z-iw&-z+iw
  \end{pmatrix}^{-1}.
\end{align*}
Define $x=u-1$ and $y=v$. The equations \eqref{u}-\eqref{v} become
\begin{align*}
  {d\over dt}\begin{pmatrix}
    x\\y
  \end{pmatrix}
  =J_+\begin{pmatrix}
    x\\y
  \end{pmatrix}+\begin{pmatrix}
    -1\\1
  \end{pmatrix}[(p+2qv)xy+quy^2+qxy^2]+\begin{pmatrix}
    -1\\0
  \end{pmatrix}bx^2.
\end{align*}
Next, we set
\begin{equation}\label{xy}
  \begin{pmatrix}
    x\\y
  \end{pmatrix}
  =\begin{pmatrix}
    1+z&1+z\\
    -z-iw&-z+iw
  \end{pmatrix}\begin{pmatrix}
    \a\\\ba
  \end{pmatrix}
  =\begin{pmatrix}
    (1+z)(\a+\ba)\\
    (-z-iw)\a+(-z+iw)\ba
  \end{pmatrix}.
\end{equation}
It then follows from the above equations and
$$\begin{pmatrix}
    1+z&1+z\\
    -z-iw&-z+iw
  \end{pmatrix}^{-1}
  ={1\over2iw(1+z)}\begin{pmatrix}
    -z+iw&-1-z\\
    z+iw&1+z
  \end{pmatrix}$$
that
\begin{align*}
  {d\over dt}\begin{pmatrix}
    \a\\\ba
  \end{pmatrix}
  =&\begin{pmatrix}
    iw\a\\-iw\ba
  \end{pmatrix}+{1\over2iw(1+z)}\begin{pmatrix}
    -1-iw\\1-iw
  \end{pmatrix}[(p+2qv)xy+quy^2+qxy^2]
  \\&~~+{1\over2iw(1+z)}\begin{pmatrix}
    z-iw\\-z-iw
  \end{pmatrix}bx^2.
\end{align*}
Note that the equation for $\ba'(t)$ is just the conjugate of the equation for $\a'(t)$.
We only need to investigate the first equation
\begin{equation}
  {d\a\over dt}=iw\a+{(-1-iw)[(p+2qv)xy+quy^2+qxy^2]+(z-iw)bx^2\over2iw(1+z)}.
\end{equation}
In view of \eqref{sz} and \eqref{xy}, we obtain
$$p+2qv=s(1+z),~qu={sz(z+1)\over w^2+z^2},~q={s^2z(z+1)\over w^2+z^2},~b=zs,$$
and
\begin{align*}
  x^2&=(1+z)^2(\a^2+2\a\ba+\ba^2),\\
  xy&=-(1+z)[(z+iw)\a^2+2z\a\ba+(z-iw)\ba^2],\\
  y^2&=(z^2-w^2+2iwz)\a^2+2(z^2+w^2)\a\ba+(z^2-w^2-2iwz)\ba^2,\\
  xy^2&=(1+z)[2(z^2+w^2)+(z^2-w^2+2iwz)]\a^2\ba+\cdots.
\end{align*}
A further calculation gives
\begin{align*}
  {(p+2qv)xy+quy^2+qxy^2\over s(1+z)}=&
  {-z^4-w^2z(z+2)+iw[z^2(1-z)-w^2(1+z)]\over w^2+z^2}\a^2
  \\-2sz\a\ba&
  +{sz(z+1)(3z^2+w^2+2iwz)\over w^2+z^2}\a^2\ba+\cdots,
\end{align*}
and
\begin{align*}
  \\&{(-1-iw)[(p+2qv)xy+quy^2+qxy^2]\over s(1+z)}\\=&
  {z^4+w^2(2z+2z^2-z^3)-w^4(1+z)+iw[z^4+z^3-z^2+w^2(z^2+3z+1)]\over w^2+z^2}\a^2
  \\&~~+2sz(1+iw)\a\ba
  \\&~~+{sz(z+1)[-3z^2+w^2(2z-1)-iw(2z+3z^2+w^2)]\over w^2+z^2}\a^2\ba+\cdots,
\end{align*}
and
$${(z-iw)bx^2\over s(1+z)}=z(1+z)[(z-iw)\a^2+2(z-iw)\a\ba]+\cdots,$$
where we only list the terms involving $\a^2$, $\a\ba$ and $\a^2\ba$.
Consequently,
$${d\a\over dt}=iw\a+{g_{20}\over2}\a^2+g_{11}\a\ba+{g_{02}\over2}\ba^2+{g_{21}\over2}\a^2\a+\cdots,$$
where
\begin{align*}
  {g_{20}\over2}&={s[z^4(z+2)+w^2(3z^2+2z)-w^4(z+1)+iw(-z^2+w^2+2zw^2)]\over2iw(w^2+z^2)},\\
  g_{11}&={sz(z^2+2z-iw)\over iw},\\
  {g_{21}\over2}&={s^2z(z+1)[-3z^2+w^2(2z-1)-iw(3z^2+2z+w^2]\over 2iw(w^2+z^2)}.
\end{align*}
Finally, we calculate the first Lyapunov coefficient \cite[(3.20)]{Ku04} as
\begin{align}
  l_1:=\re(ig_{20}g_{11}+wg_{21})=-{s^2z^2(z+1)(z+2)\over w}<0.
\end{align}
This implies that the Hopf bifurcation at $q=q_h$ is always supercritical; namely, the periodic solutions bifurcated from the Hopf bifurcation point $q=q_h$ are locally asymptotically stable.
We conclude this section with the following statement.
\begin{prop}\label{prop-H}
  Assume either (i) $p>1$ or (ii) $p\le1$ and $b>[3(1-p)+\sqrt{(1-p)(9-p)}]/4$.
  Let $q_h$ be defined as in \eqref{qh}.
  As $q$ crosses $q_h$ from left to right, the positive equilibrium $E_+$ loses its stability, and stable periodic solutions exist for $q>q_h$ and $q$ is sufficiently close to $q_h$.
\end{prop}

\section{Main results}
For reader's convenience, we return to our original system \eqref{U}-\eqref{V} and summarize our main results in the following theorem.
\begin{thm}
  Consider the non-scaled system \eqref{U}-\eqref{V}. The trivial equilibrium $E_0=(0,0)$ is always unstable.
  Let $R_0=CPK/D$ be the basic reproduction number of predator. The predator-free equilibrium $E_1=(K,0)$ is locally asymptotically stable when $R_0<1$ and unstable when $R_0>1$. For the critical case $R_0=1$, $E_1$ is locally asymptotically stable when $BC^2QK^2\le D^2$ and unstable when $BC^2QK^2>D^2$.

  If $R_0>1$, system \eqref{U}-\eqref{V} possesses a unique positive equilibrium $E_+=(U_+,V_+)$ which is locally asymptotically stable when $C^2QK^2<(CPK+B)^2/(CPK+B-D)$ and unstable when $C^2QK^2>(CPK+B)^2/(CPK+B-D)$. Moreover, a supercritical Hopf bifurcation occurs at $C^2QK^2=(CPK+B)^2/(CPK+B-D)$.

  If $R_0\le1$, system \eqref{U}-\eqref{V} has no positive equilibrium when
  $
    8BC^2QK^2<(9D-CPK)\sqrt{(D-CPK)(9D-CPK)}+27D^2-18DCPK-C^2P^2K^2,
  $
  and two positive equilibrium $E_\pm=(U_\pm,V_\pm)$ with $U_+<U_-<K$ when
  $
    8BC^2QK^2>(9D-CPK)\sqrt{(D-CPK)(9D-CPK)}+27D^2-18DCPK-C^2P^2K^2.
  $
  In the latter case, $E_-$ is always unstable, and $E_+$ is locally asymptotically stable if $4B>3(D-CPK)+\sqrt{(D-CPK)(9D-CPK)}$ and $C^2QK^2<(CPK+B)^2/(CPK+B-D)$, and $E_+$ is unstable if either $4B<3(D-CPK)+\sqrt{(D-CPK)(9D-CPK)}$ or $C^2QK^2>(CPK+B)^2/(CPK+B-D)$.
  Furthermore, a supercritical Hopf bifurcation occurs at $C^2QK^2=(CPK+B)^2/(CPK+B-D)$ if $4B>3(D-CPK)+\sqrt{(D-CPK)(9D-CPK)}$.
\end{thm}
\begin{proof}
In view of \eqref{bpq}, the condition $p>1$ (resp. $p<1$) is equivalent with $R_0>1$ (resp. $R_0<1$).
We also calculate $bq=BC^2QK^2/D^2$.
The first part of the theorem is the same as Proposition \ref{prop-E0}.

We rewrite $q_h$ in \eqref{qh} as
$$q_h={(CPK+B)^2\over D(CPK+B-D)},$$
provided $CPK+B>D$.
Hence, the condition $q<q_h$ (reps. $q>q_h$) is equivalent with $C^2QK^2<(CPK+B)^2/(CPK+B-D)$ (resp. $C^2QK^2>(CPK+B)^2/(CPK+B-D)$).
The second part of the theorem follows from Propositions \ref{prop-E}, \ref{prop-s+} and \ref{prop-H}.

Finally, the condition $b>[3(1-p)+\sqrt{(1-p)(9-p)}]/4$ is equivalent with $4B>3(D-CPK)+\sqrt{(D-CPK)(9D-CPK)}$ and the condition $b<[3(1-p)+\sqrt{(1-p)(9-p)}]/4$ is equivalent with $4B<3(D-CPK)+\sqrt{(D-CPK)(9D-CPK)}$.
We can also rewrite $q_0$ in \eqref{q0} as
$${(9D-CPK)\sqrt{(D-CPK)(9D-CPK)}+27D^2-18DCPK-C^2P^2K^2\over8BD}.$$
A combination of Propositions \ref{prop-E}, \ref{prop-s-} and \ref{prop-H} gives the third part of the theorem.
\end{proof}

{
\section{Numerical simulations}
In this section, we conduct numerical simulations to illustrate and verify our theoretical result.
For simplicity, we only consider the dimensionless system \eqref{u}-\eqref{v} with $p>1$.
In the simulation, we fix $p=1.5$ and $b=0.5$. It follows from \eqref{qh} that the Hopf bifurcation value for the bifurcation parameter $q$ is
$$q_h={(p+b)^2\over p+b-1}=4.$$
First, we choose $q=3.99$. Our result shows that the unique positive equilibrium $E_+$ is locally asymptotically stable.
From simulation we observe that the solution with initial condition $(0.4,0.1)$ converges to $E_+\approx(0.5,0.125)$; see Figure \ref{fig1}.

\begin{figure}[htp]
  \centering
  \includegraphics[height=0.5\textheight,width=0.9\textwidth]{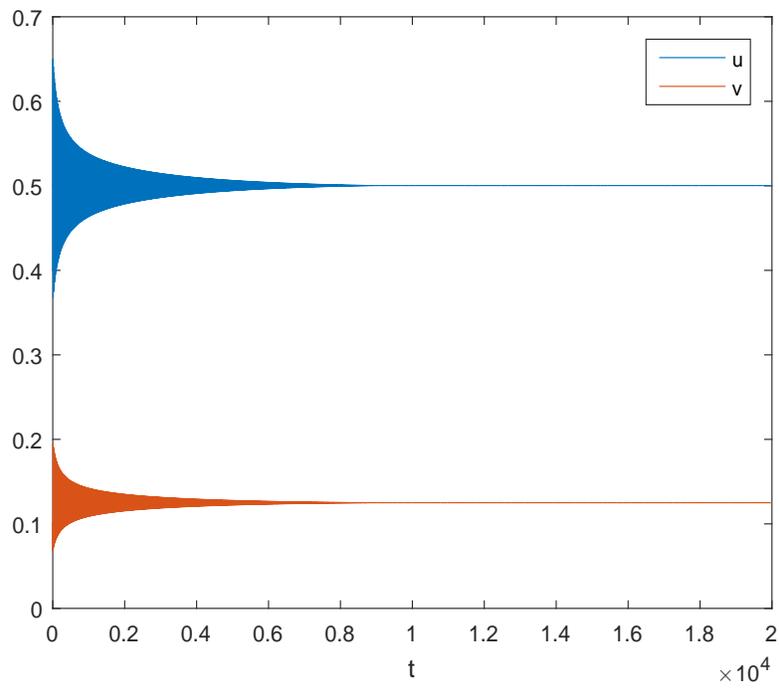}
  \caption{The case $q<q_h$.}
  \label{fig1}
\end{figure}

Next, we choose $q=4.01$. Our result shows that the unique positive equilibrium $E_+$ is unstable and stable periodic solutions exist for $q>q_h$ and $q$ is sufficiently close to $q_h$.
From simulation we observe that the solution with initial condition $(0.4,0.1)$ converges to a limit cycle; see Figure \ref{fig2}.

\begin{figure}[htp]
  \centering
  \includegraphics[height=0.5\textheight,width=0.9\textwidth]{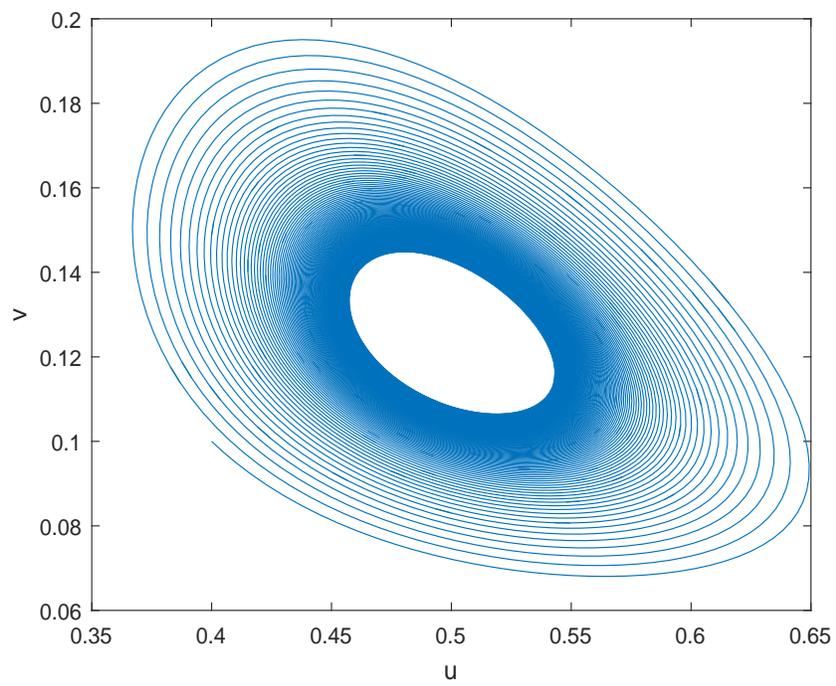}
  \caption{The case $q>q_h$.}
  \label{fig2}
\end{figure}
}

\section{Conclusion and discussion}
In this work, we provide a rigorous treatment of cooperative predation model proposed in \cite{AH17}. We obtain existence criteria and stability conditions of the positive equilibria. Especially, we prove the bistability phenomenon observed from numerical simulations in \cite{AH17}. By a careful and nontrivial computation of first Lyapunov coefficient, we demonstrate that the Hopf bifurcation is always supercritical; namely, the periodic solutions bifurcated from the Hopf bifurcation point are locally asymptotically stable.

It is worth to mention that the equation $C^2QK^2=(CPK+B)^2/(CPK+B-D)$ has exactly one solution for $Q$ when $CPK+B>D$. This means that there exists at most one Hopf bifurcation point when we use cooperative predation coefficient $Q$ as the bifurcation parameter. However, if we fix the other parameters and regard the per capita birth rate $B$ as the bifurcation parameter, the equation $C^2QK^2=(CPK+B)^2/(CPK+B-D)$ may have two distinct solutions, which indicates the existence of two Hopf bifurcation points.

Finally, we emphasize that all of our results are local: the equilibira are locally asymptotically stable under certain conditions; and the existence and stability results of periodic solutions are valid in a local and small neighborhood of the Hopf bifurcation point. It would be of more interests---and also more challenging---to investigate the global dynamics of cooperative predation model and study the global Hopf bifurcation branch. We leave this as an open problem.

\end{document}